\documentclass{birkjour}

\usepackage{amssymb}

\newtheorem{thm}{Theorem}[section]

\newtheorem{lem}[thm]{Lemma}
\newtheorem{prop}[thm]{Proposition}
\theoremstyle{definition}
\newtheorem{defn}[thm]{Definition}
\theoremstyle{remark}
\newtheorem{rem}[thm]{Remark}
\newtheorem*{ex}{Example}
\numberwithin{equation}{section}

\begin{document}

\title{Construction of a tensor product algebra with a multicentric functional calculus}

\author{Diana Andrei}

\address{%
	Aalto University\\
	Department of Mathematics and Systems Analysis\\
	P.O. Box 11100\\
	Otakaari 1M, Espoo\\
	FI-00076 Aalto, Finland}

\email{diana.andrei@aalto.fi}




\date{May 1, 2021}

\begin{abstract}
	In the multicentric calculus one takes a polynomial with simple roots as a new global variable and replaces  scalar functions $\varphi$ by functions $f$ taking values in $\mathbb{C}^d$  with $d$ the degree of the polynomial leading to  an efficient holomorphic functional calculus for bounded operators. This calculus was extended for non-holomorphic functions, so that if $A$ is not diagonalizable one can find a $p$ such that $p(A)$ is diagonalizable and apply the calculus to all matrices. This was done in \cite{oam} by creating a Banach algebra for $\mathbb{C}^d$-valued continuous functions in such a way that the original functions $\varphi$ appear as Gelfand transforms of $f.$. In this paper we consider constructing a Banach algebra for  functions  from $\mathbb{C}^2$ into  $\mathbb{C}^{d_1 \times d_2}$ which then likewise leads to a functional calculus for commuting pairs of matrices.	
\end{abstract}

\maketitle

\section{Introduction}

Let $z\in \mathbb{C}$ and $p(z)$ be a monic polynomial with distinct roots $\Lambda= \lbrace \lambda_i \rbrace,$ for $i=1,\dots ,d.$ In multicentric holomorphic calculus \cite{jfa}, one represents the holomorphic function $\varphi$ using a new polynomial variable $w=p(z),$ so that functions $\varphi (z)$ are represented with the help of vector-valued functions $f,$ mapping a compact $M\subset \mathbb{C}$ into $\mathbb{C}^{d}$ by 
$$
w \mapsto f(w) \in \mathbb{C}^{d}.
$$ 
If $\varphi$ is holomorphic then so if $f.$ This calculus is well defined for matrices that are diagonalizable. If we have non-diagonalizable matrix, $A,$ then we can still make use of the calculus by simply finding a polynomial $p$ such that $p(A)$ is diagonalizable.

Olavi Nevanlinna in \cite{oam} shows how there is a simple way to parametrize continuous functions which slow down at those places where extra smoothness is needed. This allows a functional calculus which agrees with the holomorphic functional calculus if applied to holomorphic functions, but is defined for much wider class of functions.

The holomorphic calculus was extended to $n-$tuples of commuting operators in \cite{aot}. It was then natural to ask what is the largest class of functions in original variable if the vector valued functions are only continuous. For that, one needs a Banach algebra structure which was worked out in \cite{oam} for single variable. At the end the author shows that the original function $\varphi$ is the Gelfand transformation of $f.$

The purpose of this paper is to extend the calculus in \cite{oam} for pairs of commuting operators. Note that we only consider the commuting case, but there can also be value in the non-commuting case and for this one can check \cite{oam2}. The difficulties in moving to a pair of variables is not in lifting commuting operators, but finding a Banach algebra which can then be identified with the space of continuous two variable functions. We work this out by finding a suitable product and operator norm which make the vector space of continuous functions a Banach algebra. Then another difficulty is finding the suitable cross norm so that the Banach algebra will be identified with the tensor product of single variable Banach algebras. That is, being able to identify the Banach space of continuous two variable functions with the tensor product of continuous one variable Banach spaces. Finally, because of this identifications we are able to apply the Gelfand theory and to find the characters of the Banach algebra. It then turns out that this represents, like in \cite{oam}, the original functions as Gelfand transforms.

Let's now shortly recall the $1-$variable case following \cite{oam}. Note that we are using throughout the paper the $w$ variable in the $M$ domain, and the $z$ variable in the $K=p^{-1}(M)$ domain.

Let $\sigma_{ij}$ be given by
$$
\sigma_{ij}= \frac{1}{p(\lambda_j)}\frac{1}{\lambda_i-\lambda_j}.
$$
Now for $w$ fixed complex and $\sigma_{ij}$ as well, $e_i \in \mathbb{C}^{d}$ is the usual coordinate vector, defined by 
$$
e_i \circledcirc e_i = e_i - w \sum_{j\neq i} (\sigma_{ij} e_i + \sigma_{ji} e_j )
$$
and for $j\neq i$ 
$$
e_i \circledcirc e_j = w (\sigma_{ij}e_i + \sigma_{ji}e_j).
$$
The product $\circledcirc$, called 'polyproduct', is commutative and denote $1= \sum_{i=1}^{d}e_i,$ so it is straightforward that we have $1 \circledcirc e_i = e_i. $ Then these coordinate vectors become basis elements in the Banach algebra and an element in the algebra can be written in the form 
$$
f= \sum_{i=1}^{d} f_i(w)e_i.
$$
Considering the dual of $\mathbb{C}^{d},$ the functionals $\eta$ are of the form 
$$
\eta : a= \sum_i \alpha_i e_i \mapsto \sum_i \eta_i \alpha_i
$$
and requiring $\eta (1)=1$ means $\sum_i \eta_i =1.$ The functional $\eta$ is a {\it character} if, in addition, we have
$$
\eta(a \circledcirc b) = \eta (a) \eta (b).
$$

If we assume we are given continuous function $f$ mapping a compact $M \subset \mathbb{C} $ into $\mathbb{C}^{d},$ it determines a unique function $\varphi$ on $K=p^{-1}(M)$ by setting
\begin{equation}
	\label{varphi}
	\varphi(z) = \sum_{j=1}^{d} \delta_{j} (z) f_{j}(p(z)), \quad\text{for } z\in K
\end{equation}
and thus $\varphi$ is given on $K$ by a {\em multicentric representation} denoted in short
$$
\varphi = \mathcal{L} f.
$$

Consider now $w=p(z)$ as a new variable. After defining an operator norm, $\|\cdot \|,$ the vector space $C(M)^{d}$ of continuous functions $f$ from a compact $M\subset \mathbb{C}$ into $\mathbb{C}^{d},$ with the operator norm $\| f\|$ and the product $\circledcirc$ becomes a Banach algebra denoted by $C_\Lambda (M).$ Then the function $\varphi$ in (\ref{varphi}) appears as the Gelfand transform in the new algebra $C_\Lambda (M).$ 

When we move to two variables, there are technical parts on how to define the algebra, but the key is to get the norm in this algebra so that it becomes a Banach algebra and it represents all continuous functions. After that one still needs to find all characters, which then allows one to look at $\varphi$ as the Gelfand transformation of $f.$

\section{Construction of Banach algebra}

\subsection{Outline of {\boldmath $1$} variable case}
Recall for \cite{oam} the following notions and results.
\begin{defn}
	\label{def_polyprod1}
	Let $f$ and $g$ be pointwise defined functions from $M\subset \mathbb{C}$ into $\mathbb{C}^{d}.$ Then their 'polyproduct' $f\circledcirc g$ is a function defined on $M,$ taking values in $\mathbb{C}^{d}$ such that
	$$
	(f\circledcirc g)(w)= (f\circ g)(w)- w(L\circ \Box f(w)\circ \Box g(w))l,
	$$
	where $L$ is a matrix which has zero diagonal and $L_{ij}=1/(\lambda_i-\lambda_j)$ for $i\neq j,$  $l$ is a vector in $ \mathbb{C}^{d}$ which has components $l_j=1/p'(\lambda_j),$ $\circ$ is the Hadamard (or Schur, elementwise) product and for $a\in \mathbb{C}^{d}$ we set
	$$
	\Box : a \mapsto \Box a= 
	\begin{pmatrix}
		0 & a_1-a_2 & \dots & a_1-a_d \\
		a_2-a_1 & 0 & \dots & a_2-a_d \\
		\dots & \dots & \dots & \dots \\
		a_d-a_1 & \dots & a_d-a_{d-1} & 0 
	\end{pmatrix}
	$$
	and call it boxing the vector $a.$ 
\end{defn}

The following short notation will be use for simplicity:
$$
f\circledcirc g=f\circ g - w(L\circ \Box f\circ \Box g)l.
$$
For the powers we write $f^{n}=f\circledcirc f^{n-1}$ and the inverse as $f^{-1}$ where it exists: $f \circledcirc f^{-1}= {\bf 1},$ $({\bf 1}=(1,\dots ,1)^{t}\in \mathbb{C}^{d}, \text{ is the unit in the algebra}).$

\begin{prop}\label{prop2.2}
	The vector space of functions 
	$$
	f:M \subset \mathbb{C} \rightarrow \mathbb{C}^{d}
	$$
	equipped with the product $\circledcirc$ becomes a complex commutative algebra with ${\bf 1}$ as the unit.
\end{prop}

\begin{thm}
	\label{thm2.3}
	Let $f$ and $g$ be defined in $M $ and $K=p^{-1}(M).$ Then if $\varphi$ and $\psi$ are functions defined on $K$ by $\varphi=\mathcal{L}f$ and $\psi=\mathcal{L}g,$ then $\varphi\psi$ is given by
	$$
	\varphi\psi=\mathcal{L}(f\circledcirc g).
	$$
\end{thm}

We thus want to define the 'polyproduct' when having two commuting variables and then define the corresponding Banach algebra.
To that end, consider that we have two vector spaces and we take one vector from each, so there are finite sums of them, and then we define a product of these. Since the product will be defined on any two vectors and it is a finite linear combination of them, then we can extend to an algebra on the tensor product of the vector spaces. With a suitable norm this algebra will then become a Banach algebra and the Gelfand theory can be successfully applied into the multicentric setting for two commuting operators. 
This is worked out after recalling a few more facts from the literature.

\subsection{ Tensor product of two spaces }

Let $X$ and $Y$ be two linear vector spaces. The space consisting of linear combinations of elements in these spaces is then the tensor product space $X \otimes Y.$ We write $X \otimes_a Y$ and refer to the algebraic tensor product of $X$ and $Y,$ meaning that there are well defined products in each vector space that induces a well defined product in the tensor product space. 

Recall from \cite{ap} the following useful notions and results.

\begin{defn}\label{projectivenorm}
	Let $X$ and $Y$ be Banach spaces and $X \otimes_{ a} Y$ be the algebraic tensor product of $X$ and $Y$ as linear vector spaces over $\mathbb{C}.$ If for $u \in X \otimes_{ a} Y$ one defines
	$$
	\| u \|_{\gamma} = \inf \lbrace \sum_{i=1}^{n} \| x_i \| \|y_i \|: x_1, \dots ,x_n \in X, y_1, \dots ,y_n \in Y, u= \sum_{i=1}^{n}x_i \otimes y_i \rbrace,
	$$ 
	then $\| \cdot \|_{\gamma}$ is a norm on $X \otimes_{ a} Y.$ The completion of $X \otimes_{ a} Y$ is the projective tensor product of $X$ and $Y$ and is denoted $X \hat{\otimes}_{\gamma} Y.$
	
\end{defn}

\begin{defn}\label{injectivenorm}
	Let $X$ and $Y$ be Banach spaces and $X \otimes_{ a} Y$ be the algebraic tensor product of $X$ and $Y$ as linear vector spaces over $\mathbb{C}.$ Let $X^{*}$ and $Y^{*}$ be the dual of $X$ and $Y,$ respectively. If one defines for $u\in X \otimes_{ a} Y$ the norm given by
	\begin{align}
		\| u \|_{\varepsilon} = \sup \lbrace | \sum_{i=1}^{n} \varphi( x_i) \psi(y_i ) | &: x_1, \dots ,x_n \in X, y_1, \dots ,y_n \in Y, \nonumber\\
		&\text{ } \varphi \in (X^{*})_1, \psi \in (Y^{*})_1,   u= \sum_{i=1}^{n}x_i \otimes y_i \rbrace,\nonumber
	\end{align}
	where $(X^{*})_1$ and $(Y^{*})_1$ denote the unit balls of the dual Banach spaces $X^{*}$ and $Y^{*},$ respectively, then $\| \cdot \|_\varepsilon$ is a norm on $X \otimes_{ a} Y.$ The completion of $X \otimes_{ a} Y$ is the injective tensor product of $X$ and $Y$ and is denoted $X \hat{\otimes}_{\varepsilon}Y.$
\end{defn}

The fact that these two are norms on $X \otimes_{ a} Y$ is easy to check, and it can also be found in \cite{sprR}.

Then one has the following holding true: For $X$ and $Y$ compact Hausdorff spaces 
\begin{equation}
	C(X) \hat{\otimes}_{\varepsilon} C(Y)=C(X\times Y) \label{eq2.1}
\end{equation} 
\begin{equation}
	C(X) \hat{\otimes}_{\gamma} C(Y)=C(X\times Y) \text{ iff } X \text{ or } Y \text{ is finite}. \label{eq2.2}
\end{equation} $($ Here $C(X)$ means the space of continuous functions on $X.)$ 
The statement given by (\ref{eq2.1}) appears in \cite{sprR} as well and for a better understanding one can check it there.

\subsection{Deriving the algebra structure}

For $M_1, M_2 \subset \mathbb{C}$ compact sets, we consider continuous functions $f$ mapping $M_1 \times M_2$ into $\mathbb{C}^{d_1 \times d_2}$ so that we can define an algebra structure. Since scalar multiplication and summation are obvious, we only need to construct the multiplication in a consistent way. 

To that end we rewrite notations and formulas from \cite{oam} so that computations can be extended to two commuting variables.

Let $w_1,w_2$ be fixed complex numbers and let $\sigma_{1,jl}, \sigma_{2,km}$ be given by
\begin{align}
	\sigma_{1,jl} &= \frac{1}{p_1'(\lambda_l)}\frac{1}{\lambda_j-\lambda_l} \label{sigma1jl} \\
	\sigma_{2,km} &= \frac{1}{p_2'(\mu_m)}\frac{1}{\mu_k-\mu_m} \label{sigma2km}.
\end{align}

Let $e_{1,j} \in \mathbb{C}^{d_1}$ and $e_{2,k}\in \mathbb{C}^{d_2}$ be the usual coordinate vectors. While these coordinate vectors are constant vectors in $\mathbb{C}^{d_1}$ and $\mathbb{C}^{d_2},$ respectively, the 'polyproduct' yields a function in $w_i'$s pointwisely as follows. Note that we use the term "basis elements" instead of "coordinate vectors" when we move to tensor algebras, or two variable setup.

\begin{defn}\label{defbasis}
	Define in $\mathbb{C}^{d_1}$ 
	\begin{align}
		e_{1,j} \circledcirc e_{1,j} &= e_{1,j}-w_1 \sum_{1=l\neq j}^{d_1}( \sigma_{1,jl} e_{1,j} + \sigma_{1,lj} e_{1,l} ), \nonumber\\
		e_{1,j} \circledcirc e_{1,l} &= w_1 (\sigma_{1,jl} e_{1,j} + \sigma_{1,lj} e_{1,l}), \quad \text{ for } j \neq l, \nonumber
	\end{align}
	while in $\mathbb{C}^{d_2}$
	\begin{align}
		e_{2,k} \circledcirc e_{2,k} &= e_{2,k}-w_2 \sum_{1=m\neq k}^{d_2}( \sigma_{2,km} e_{2,k} + \sigma_{2,mk} e_{2,m} ), \nonumber\\
		e_{2,k} \circledcirc e_{2,m} &= w_2 (\sigma_{2,km} e_{2,k} + \sigma_{2,mk} e_{2,m}), \quad \text{ for } k\neq m. \nonumber
	\end{align}	
\end{defn}

On the tensor product of $C(M_1,\mathbb{C}^{d_1})$ and $C(M_2, \mathbb{C}^{d_2})$ (i.e continuous functions from $M_i$ into $\mathbb{C}^{d_i}$ for $i=1,2$ ) as vector spaces, there is a unique natural definition of the product that makes it into an algebra. More explicitly, $C(M_1,\mathbb{C}^{d_1}),$ and $C(M_2, \mathbb{C}^{d_2})$ together with the multiplication and operator norm are the algebras $C_{\Lambda_1}(M_1)$ and $C_{\Lambda_2}(M_2),$ and so the tensor product $C_{\Lambda_1} (M_1) \otimes C_{\Lambda_2}(M_2)$ will give the new algebra structure for the two variable case, which is denoted $C_{\Lambda_1 \times \Lambda_2}(M_1 \times M_2).$

\begin{defn}\label{basisformula}
	Let $e_{1,j} \otimes e_{2,k}$ be the basis elements in $C_{\Lambda_1} (M_1) \otimes C_{\Lambda_2}(M_2)$. There is a natural way to define a product in the tensor vector space $C_{\Lambda_1} (M_1) \otimes C_{\Lambda_2}(M_2)$ that makes it into an algebra. This product is induced by the product in the two algebras, $C_{\Lambda_1} (M_1)$ and $C_{\Lambda_2} (M_2)$ so we can define in $\mathbb{C}^{d_1 \times d_2}$ the multiplication of the basis elements as follows 
	\begin{equation}\label{basisproduct}
		(e_{1,j} \otimes e_{2,k} ) \circledcirc ( e_{1,l} \otimes e_{2,m} ) = (e_{1,j} \circledcirc e_{1,l} ) \otimes ( e_{2,k} \circledcirc e_{2,m} ) .
	\end{equation}
\end{defn}

We carefully look at the 'polyproduct' $\circledcirc$ in the vector space $ C_{\Lambda_1} (M_1) \otimes C_{\Lambda_2}(M_2) $ and later define a norm in this tensor space so at the end we can complete the space such that it becomes a Banach algebra, denoted $C_{\Lambda_1 \times \Lambda_2}(M_1 \times M_2).$

So far the above definition looks rather abstract so we need to expand it for a clear read, as follows. A formula for the product in the algebra, $(f \circledcirc g )(w_1,w_2),$ is given after we define a few needed notations. Set
\begin{align*}
	& L_1 \text{  a diagonal matrix with zero diagonal and } L_{1,jl}=\frac{1}{\lambda_j-\lambda_l} \\
	& L_2 \text{  a diagonal matrix with zero diagonal and } L_{2,km}=\frac{1}{\mu_k-\mu_m} \\
	& l_1 \text{  a vector in } \mathbb{C}^{d_1} \text{ with elements } l_{1,l}= \frac{1}{p'_1(\lambda_l)} \\
	& l_2 \text{  a vector in } \mathbb{C}^{d_2} \text{ with elements } l_{2,m}= \frac{1}{p'_2(\mu_m)} .
\end{align*}
By $L_1^{T}, L_2^{T}, l_1^{T}$ and $l_2^{T}$ we mean the transpose of $L_1, L_2, l_1$ and $l_2$ respectively.

Define for $f\in \mathbb{C}^{d_1\times d_2}$
\begin{align*}
	\Box_k : f \mapsto \Box_k  f = 
	\begin{pmatrix}
		0 & f_{1,k}-f_{2,k} & \dots & f_{1,k}-f_{d_1,k}\\
		f_{2,k}-f_{1,k} & 0 & \dots & f_{2,k}-f_{d_1,k} \\
		\dots & \dots & \dots & \dots \\
		f_{d_1,k}-f_{1,k} & \dots & f_{d_1,k}-f_{d_1-1,k} & 0 
	\end{pmatrix}
\end{align*}
the boxing of the $k-$th column of $f(w_1,w_2),$
and
\begin{align*}
	\text{}_j\Box : f \mapsto \text{}_j \Box  f = 
	\begin{pmatrix}
		0 & f_{j,1}-f_{j,2} & \dots & f_{j,1}-f_{j,d_2}\\
		f_{j,2}-f_{j,1} & 0 & \dots & f_{j,2}-f_{j,d_2} \\
		\dots & \dots & \dots & \dots \\
		f_{j,d_2}-f_{j,1} & \dots & f_{j,d_2}-f_{j,d_2-1} & 0 
	\end{pmatrix}
\end{align*}
the boxing of the transpose of the $j-$th row of $f(w_1,w_2).$

\begin{defn}	
	Let $f$ and $g$ be pointwisely defined functions from $M_1 \times M_2 \subset \mathbb{C}^{2}$ into $\mathbb{C}^{d_1 \times d_2}.$ Then the 'polyproduct' $f \circledcirc g$ is a function defined on $M_1 \times M_2$ taking values in $\mathbb{C}^{d_1\times d_2}$ such that elementwise we have
	\small{
		\begin{align}
			\label{polyprod}
			(f \circledcirc g)_{jk}(w_1,w_2) &= (f_{j,k} \circ g_{j,k})(w_1,w_2) \nonumber\\
			&\quad  - w_1 (L_1 \circ \Box_k f(w_1,w_2) \circ \Box_k g(w_1,w_2) )l_1 \nonumber \\
			&\quad  - w_2 l_2^{T}(L_2^{T} \circ \text{ }_j\Box f^{T}(w_1,w_2) \circ \text{ }_j\Box g^{T}(w_1,w_2)) \nonumber \\
			&\quad  + w_1 w_2  (l_2 \otimes l_1)^{T} ((L_2^{T}\otimes \text{ }_j L_1^{T}) \circ \boxtimes_j f^{T}(w_1,w_2) \circ \boxtimes_j g^{T}(w_1,w_2)) \nonumber\\
			\text{}
	\end{align}}
\noindent where by $\text{ }_j L_1^{T}$ we mean the transpose of the $j-$th line in $L_1,$ i.e. a vector in $\mathbb{C}^{d_1}$, and for $f\in \mathbb{C}^{d_1\times d_2}$ we define $ \boxtimes_j f^{T} $ as the matrix in $\mathbb{C}^{d_1^{2} \times d_2}$with $d_2$ columns, denoted ${\rm col}_k,$ for $k=1, \dots ,d_2,$ and each column is a vector in $\mathbb{C}^{d_1^{2}}$:
	
	\begin{equation}
		{\rm col}_1 = 
		\begin{pmatrix}
			0 \\
			0 \\
			\dots \\
			0 \\
			(f_{j,1}-f_{j,2})-(f_{1,1}-f_{1,2}) \\
			(f_{j,1}-f_{j,2})-(f_{2,1}-f_{2,2}) \\
			\dots \\
			(f_{j,1}-f_{j,2})-(f_{d_1,1}-f_{d_1,2}) \\
			(f_{j,1}-f_{j,3})-(f_{1,1}-f_{1,3}) \\
			(f_{j,1}-f_{j,3})-(f_{2,1}-f_{2,3}) \\
			\dots \\
			(f_{j,1}-f_{j,3})-(f_{d_1,1}-f_{d_1,3}) \\
			\dots \dots \dots \dots \\
			(f_{j,1}-f_{j,d_2})-(f_{1,1}-f_{1,d_2}) \\
			(f_{j,1}-f_{j,d_2})-(f_{2,1}-f_{2,d_2}) \\
			\dots \\
			(f_{j,1}-f_{j,d_2})-(f_{d_1,1}-f_{d_1,d_2}) \\
		\end{pmatrix}, 
		{\rm col}_2 = 
		\begin{pmatrix}
			(f_{j,2}-f_{j,1})-(f_{1,2}-f_{1,1}) \\
			(f_{j,2}-f_{j,1})-(f_{2,2}-f_{2,1}) \\
			\dots \\
			(f_{j,2}-f_{j,1})-(f_{d_1,2}-f_{d_1,1}) \\
			0 \\
			0 \\
			\dots \\
			0 \\
			(f_{j,2}-f_{j,3})-(f_{1,2}-f_{1,3}) \\
			(f_{j,2}-f_{j,3})-(f_{2,2}-f_{2,3}) \\
			\dots \\
			(f_{j,2}-f_{j,3})-(f_{d_1,2}-f_{d_1,3}) \\
			\dots \dots \dots \dots \\
			(f_{j,2}-f_{j,d_2})-(f_{1,2}-f_{1,d_2}) \\
			(f_{j,2}-f_{j,d_2})-(f_{2,2}-f_{2,d_2}) \\
			\dots \\
			(f_{j,2}-f_{j,d_2})-(f_{d_1,2}-f_{d_1,d_2}) \\
		\end{pmatrix} 
		\nonumber
	\end{equation}
	and the last column will be
	\begin{equation}
		{\rm col}_{d_2} = 
		\begin{pmatrix}
			(f_{j,d_2}-f_{j,1})-(f_{1,d_2}-f_{1,1}) \\
			(f_{j,d_2}-f_{j,1})-(f_{2,d_2}-f_{2,1}) \\
			\dots \\
			(f_{j,d_2}-f_{j,1})-(f_{d_1,d_2}-f_{d_1,1}) \\
			(f_{j,d_2}-f_{j,2})-(f_{1,d_2}-f_{1,2}) \\
			(f_{j,d_2}-f_{j,2})-(f_{2,d_2}-f_{2,2}) \\
			\dots \\
			(f_{j,d_2}-f_{j,2})-(f_{d_1,d_2}-f_{d_1,2}) \\
			\dots \dots \dots \dots \\
			(f_{j,d_2}-f_{j,d_2-1})-(f_{1,d_2}-f_{1,d_2-1}) \\
			(f_{j,d_2}-f_{j,d_2-1})-(f_{2,d_2}-f_{2,d_2-1}) \\
			\dots \\
			(f_{j,d_2}-f_{j,d_2-1})-(f_{d_1,d_2}-f_{d_1,d_2-1}) \\
			0 \\
			0 \\
			\dots \\
			0 \\
		\end{pmatrix} 
		\nonumber
	\end{equation}
\end{defn}

The formula of the polyproduct can easily be checked using Definition~\ref{basisformula}, as well as following the proof of the next result.

\begin{prop}\label{prop2.9}
	Let $f$ and $g$ be defined on $M_1 \times M_2$ and $K_1 \times K_2 = p_1^{-1}(M_1) \times p_2^{-1}(M_2).$ Then if $\varphi $ and $\psi$ are functions defined on $K_1 \times K_2$ by the multicentric representation
	$$
	\varphi = \mathcal{L} f, \quad \psi = \mathcal{L} g,
	$$
	then we have
	$$
	\varphi \psi = \mathcal{L} (f \circledcirc g).
	$$
\end{prop}
In other words, $\mathcal{L}$ takes the vector functions into scalar functions in such a way that it becomes an algebra homomorphism 
$$
\mathcal{L}(f \circledcirc g) = (\mathcal{L}f) (\mathcal{L}g).
$$

\begin{proof}
	Let $\varphi(z_1,z_2)$ and $\psi(z_1,z_2)$ be given in the multicentric representation as follows,
	\begin{align}
		\varphi(z_1,z_2) &= \sum_{k=1}^{d_2} \sum_{j=1}^{d_1} \delta_{2,k}(z_2) \delta_{1,j}(z_1) f_{j,k}(w_1,w_2) \label{eq1.2} \\
		\psi(z_1,z_2) &= \sum_{m=1}^{d_2} \sum_{l=1}^{d_1} \delta_{2,m}(z_2) \delta_{1,l}(z_1) g_{l,m}(w_1,w_2) \nonumber.
	\end{align}
	
	Then we can write $\varphi(z_1,z_2) \psi (z_1,z_2),$ in short,
	\begin{align}
		\varphi \psi &= \sum_{j,k,l,m} \delta_{k,2}\delta_{j,1} f_{j,k} \delta_{m,2} \delta_{l,1} g_{l,m} \nonumber \\
		&= \sum_{j,l,k,m} \delta_{j,1} \delta_{l,1} \delta_{k,2} \delta_{m,2} f_{j,k} g_{l,m} ,\label{eq2.7}
	\end{align}
	for $j,l=1,2,\dots ,d_1$ and $k,m=1,2,\dots ,d_2.$
	
	Then we have the following estimates (see \cite{oam})
	\begin{align}
		\delta_{1,j}^{2}(z_1) &= \delta_{1,j}(z_1)-w_1 \sum_{l\neq j} (\sigma_{1,jl}\delta_{1,j}(z_1)+ \sigma_{1,lj}\delta_{1,l}(z_1)) \label{eq2.8} \\
		\delta_{2,k}^{2} (z_2)&= \delta_{2,k}(z_2)-w_2 \sum_{m\neq k} (\sigma_{2,km}\delta_{1,k}(z_2)+ \sigma_{2,mk}\delta_{2,m}(z_2)),\label{eq2.9} 
	\end{align}
	while for $j \neq l$ and $k \neq m,$ respectively, we have
	\begin{align}
		\delta_{1,j}(z_1)\delta_{1,l}(z_1) &= w_1 [ \sigma_{1,jl}\delta_{1,j}(z_1)+ \sigma_{1,lj}\delta_{1,l}(z_1) ] \label{eq2.10} \\
		\delta_{2,k}(z_2)\delta_{2,m}(z_2) &= w_2 [ \sigma_{2,km}\delta_{1,k}(z_2)+ \sigma_{2,mk}\delta_{2,m}(z_2) ]\label{eq2.11} .
	\end{align}
	
	Now we use the estimates (\ref{eq2.8}) and (\ref{eq2.9}) and get for $l=j$ and $m=k,$
	\begin{align}
		\delta_{1,j}^{2} \delta_{2,k}^{2} &=(\delta_{1,j} - w_1 \sum_{l\neq j} (\sigma_{1,jl}\delta_{1,j}+ \sigma_{1,lj}\delta_{1,l}) ) \nonumber\\
		& \hspace{1cm} ( \delta_{2,k}-w_2 \sum_{m\neq k} (\sigma_{2,km}\delta_{1,k}+ \sigma_{2,mk}\delta_{2,m}) )\nonumber\\
		&= \delta_{1,j} \delta_{2,k} - w_1 \delta_{2,k} \sum_{l\neq j} (\sigma_{1,jl}\delta_{1,j}+ \sigma_{1,lj}\delta_{1,l}) \nonumber\\
		& \hspace{1.5cm} -w_2 \delta_{1,j} \sum_{m\neq k} (\sigma_{2,km}\delta_{1,k}+ \sigma_{2,mk}\delta_{2,m}) \nonumber\\
		& \hspace{1.5cm} + w_1w_2 \sum_{\scriptscriptstyle l \neq j,m\neq k}  ( \sigma_{1,jl} \sigma_{2,km} \delta_{1,j}\delta_{2,k}+ \sigma_{1,jl} \sigma_{2,mk} \delta_{1,j}\delta_{2,m} \nonumber\\
		& \hspace{3.5cm}+ \sigma_{1,lj} \sigma_{2,km} \delta_{1,l}\delta_{2,k} + \sigma_{1,lj} \sigma_{2,mk} \delta_{1,l}\delta_{2,m}) \label{eq2.12}
	\end{align}
	while for $j\neq l$ and $k=m$ we use (\ref{eq2.9}) and (\ref{eq2.10}) to get
	\begin{align}
		\delta_{1,j}\delta_{1,l}\delta_{2,k}^{2} &= ( w_1 (\sigma_{1,jl}\delta_{1,j}+ \sigma_{1,lj}\delta_{1,l}) )  ( \delta_{2,k}-w_2 \sum_{m\neq k} (\sigma_{2,km}\delta_{1,k}+ \sigma_{2,mk}\delta_{2,m}) )\nonumber\\
		&= w_1 \delta_{2,k} \ (\sigma_{1,jl}\delta_{1,j}+ \sigma_{1,lj}\delta_{1,l}) \nonumber\\
		& \hspace{0.5cm} - w_1w_2 \sum_{m\neq k}  ( \sigma_{1,jl} \sigma_{2,km} \delta_{1,j}\delta_{2,k}+ \sigma_{1,jl} \sigma_{2,mk} \delta_{1,j}\delta_{2,m} \nonumber\\
		& \hspace{2.2cm}+ \sigma_{1,lj} \sigma_{2,km} \delta_{1,l}\delta_{2,k} + \sigma_{1,lj} \sigma_{2,mk} \delta_{1,l}\delta_{2,m}), \label{eq2.13}
	\end{align}
	for $j=l$ and $k \neq m$ we use (\ref{eq2.8}) and (\ref{eq2.11}) to get
	\begin{align}
		\delta_{1,j}^{2}\delta_{2,k}\delta_{2,m} &= (\delta_{1,j} - w_1 \sum_{l\neq j} (\sigma_{1,jl}\delta_{1,j}+ \sigma_{1,lj}\delta_{1,l}) )( w_2  (\sigma_{2,km}\delta_{1,k}+ \sigma_{2,mk}\delta_{2,m}) )\nonumber\\
		&= w_2 \delta_{1,j}  (\sigma_{2,km}\delta_{1,k}+ \sigma_{2,mk}\delta_{2,m}) \nonumber\\
		& \hspace{0.5cm} - w_1w_2 \sum_{l \neq j}  ( \sigma_{1,jl} \sigma_{2,km} \delta_{1,j}\delta_{2,k}+ \sigma_{1,jl} \sigma_{2,mk} \delta_{1,j}\delta_{2,m} \nonumber\\
		& \hspace{2.2cm}+ \sigma_{1,lj} \sigma_{2,km} \delta_{1,l}\delta_{2,k} + \sigma_{1,lj} \sigma_{2,mk} \delta_{1,l}\delta_{2,m}), \label{eq2.14}
	\end{align}
	and finally for $j\neq l$ and $k\neq m$ we use (\ref{eq2.10}) and (\ref{eq2.11}) and we see that
	\begin{align}
		\delta_{1,j}\delta_{1,l}\delta_{2,k}\delta_{2,m} &= ( w_1  (\sigma_{1,jl}\delta_{1,j}+ \sigma_{1,lj}\delta_{1,l}) ) ( w_2  (\sigma_{2,km}\delta_{1,k}+ \sigma_{2,mk}\delta_{2,m}) )\nonumber\\
		&= w_1w_2  ( \sigma_{1,jl} \sigma_{2,km} \delta_{1,j}\delta_{2,k}+ \sigma_{1,jl} \sigma_{2,mk} \delta_{1,j}\delta_{2,m} \nonumber\\
		& \hspace{1.1cm}+ \sigma_{1,lj} \sigma_{2,km} \delta_{1,l}\delta_{2,k} + \sigma_{1,lj} \sigma_{2,mk} \delta_{1,l}\delta_{2,m}) .\label{eq2.15}
	\end{align}
	
	\noindent Hence (\ref{eq2.7}) becomes
	\begin{align}
		\varphi \psi &= \sum_{\scriptscriptstyle j,k} \delta_{1,j}^{2} \delta_{2,k}^{2} f_{j,k} g_{j,k}  +  \sum_{\scriptscriptstyle j,k ,l\neq j} \delta_{1,j}\delta_{1,l}\delta_{2,k}^{2} f_{j,k}g_{l,k} \nonumber\\
		&\hspace{1cm}+ \sum_{ \scriptscriptstyle j,k ,m\neq k }\delta_{1,j}^{2}\delta_{2,k}\delta_{2,m} f_{j,k}g_{j,m} + \sum_{ \scriptscriptstyle j,k ,l\neq j, m\neq k } \delta_{1,j}\delta_{1,l}\delta_{2,k}\delta_{2,m} f_{j,k}g_{l,m} \nonumber\\
		&= \sum_{\scriptscriptstyle j,k} \delta_{1,j} \delta_{2,k} \left[ f_{j,k}g_{j,k} -w_1  \sum_{\scriptscriptstyle l\neq j} \sigma_{1,jl} (f_{j,k}-f_{l,k})(g_{j,k}-g_{l,k})\right. \nonumber\\
		&\hspace{0.5cm} -w_2  \sum_{\scriptscriptstyle m\neq k} \sigma_{2,km} (f_{j,k}-f_{j,m})(g_{j,k}-g_{j,m}) \nonumber\\
		&\hspace{0.5cm} +w_1 w_2 \sum_{\scriptscriptstyle l\neq j, m\neq k} \sigma_{1,jl} \sigma_{2,km} ((f_{j,k}-f_{j,m})-(f_{l,k}-f_{l,m})) \nonumber\\
		&\hspace{4.5cm} \left. ((g_{j,k}-g_{j,m})-(g_{l,k}-g_{l,m})) \right]
		\label{eq2.16}
	\end{align}
\end{proof}

\noindent We continue with an example to see how formulas \eqref{eq2.16} / \eqref{polyprod} look like.

\begin{ex}
	\label{example2.10} Let $p_1(z_1)$ be a monic polynomial of degree $2$ with distinct roots $\lambda_1$ and $\lambda_2,$ and $p_2(z_2)$ be a monic polynomial of degree $3$ with distinct roots $\mu_1,$ $\mu_2$ and $\mu_3$. Denote $w_1=p_1(z_1)$ and $w_2=p_2(z_2).$ Then $(f \circledcirc g)(w_1,w_2)$ is a $2\times 3$ matrix with the following elements	
	\begin{align}
		(f \circledcirc g)_{11} &= f_{1,1}g_{1,1}-w_1 \sigma_{1,2}^{(1)}(f_{1,1}-f_{2,1})(g_{1,1}-g_{2,1}) \nonumber\\
		&- w_2 [ \sigma_{1,2}^{(2)}(f_{1,1}-f_{1,2})(g_{1,1}-g_{1,2})+ \sigma_{1,3}^{(2)}(f_{1,1}-f_{1,3})(g_{1,1}-g_{1,3}) ] \nonumber\\
		&+ w_1 w_2 [ \sigma_{1,2}^{(2)} \sigma_{1,2}^{(1)} ((f_{1,1}-f_{1,2})-(f_{2,1}-f_{2,2}))((g_{1,1}-g_{1,2})-(g_{2,1}-g_{2,2}))\nonumber\\
		&\text{ } + \sigma_{1,3}^{(2)} \sigma_{1,2}^{(1)} ((f_{1,1}-f_{1,3})-(f_{2,1}-f_{2,3}))((g_{1,1}-g_{1,3})-(g_{2,1}-g_{2,3}))] \nonumber\\
		(f \circledcirc g)_{12} &= f_{1,2}g_{1,2}-w_1 \sigma_{1,2}^{(1)}(f_{1,2}-f_{2,2})(g_{1,2}-g_{2,2}) \nonumber\\
		&- w_2 [ \sigma_{2,1}^{(2)}(f_{1,2}-f_{1,1})(g_{1,2}-g_{1,1})+ \sigma_{2,3}^{(2)}(f_{1,2}-f_{1,3})(g_{1,2}-g_{1,3}) ] \nonumber\\
		&+ w_1 w_2 [ \sigma_{2,1}^{(2)} \sigma_{1,2}^{(1)} ((f_{1,2}-f_{1,1})-(f_{2,2}-f_{2,1}))((g_{1,2}-g_{1,1})-(g_{2,2}-g_{2,1}))\nonumber\\
		&\text{ } + \sigma_{2,3}^{(2)} \sigma_{1,2}^{(1)} ((f_{1,2}-f_{1,3})-(f_{2,2}-f_{2,3}))((g_{1,2}-g_{1,3})-(g_{2,2}-g_{2,3}))] \nonumber 
	\end{align}
	\begin{align}
		(f \circledcirc g)_{13} &= f_{1,3}g_{1,3}-w_1 \sigma_{1,2}^{(1)}(f_{1,3}-f_{2,3})(g_{1,3}-g_{2,3}) \nonumber\\
		&- w_2 [ \sigma_{3,1}^{(2)}(f_{1,3}-f_{1,1})(g_{1,3}-g_{1,1})+ \sigma_{3,2}^{(2)}(f_{1,3}-f_{1,2})(g_{1,3}-g_{1,2}) ] \nonumber\\
		&+ w_1 w_2 [ \sigma_{3,1}^{(2)} \sigma_{1,2}^{(1)} ((f_{1,3}-f_{1,1})-(f_{2,3}-f_{2,1}))((g_{1,3}-g_{1,1})-(g_{2,3}-g_{2,1}))\nonumber\\
		&\text{ } + \sigma_{3,2}^{(2)} \sigma_{1,2}^{(1)} ((f_{1,3}-f_{1,2})-(f_{2,3}-f_{2,2}))((g_{1,3}-g_{1,2})-(g_{2,3}-g_{2,2}))] \nonumber \\
		(f \circledcirc g)_{21} &= f_{2,1}g_{2,1}-w_1 \sigma_{2,1}^{(1)}(f_{2,1}-f_{1,1})(g_{2,1}-g_{1,1}) \nonumber\\
		&- w_2 [ \sigma_{1,2}^{(2)}(f_{2,1}-f_{2,2})(g_{2,1}-g_{2,2})+ \sigma_{1,3}^{(2)}(f_{2,1}-f_{2,3})(g_{2,1}-g_{2,3}) ] \nonumber\\
		&+ w_1 w_2 [ \sigma_{1,2}^{(2)} \sigma_{2,1}^{(1)} ((f_{2,1}-f_{2,2})-(f_{1,1}-f_{1,2}))((g_{2,1}-g_{2,2})-(g_{1,1}-g_{1,2}))\nonumber\\
		&\text{ } + \sigma_{1,3}^{(2)} \sigma_{2,1}^{(1)} ((f_{2,1}-f_{2,3})-(f_{1,1}-f_{1,3}))((g_{2,1}-g_{2,3})-(g_{1,1}-g_{1,3}))] \nonumber \\
		(f \circledcirc g)_{22} &= f_{2,2}g_{2,2}-w_1 \sigma_{2,1}^{(1)}(f_{2,2}-f_{1,2})(g_{2,2}-g_{1,2}) \nonumber\\
		&- w_2 [ \sigma_{2,1}^{(2)}(f_{2,2}-f_{2,1})(g_{2,2}-g_{2,1})+ \sigma_{2,3}^{(2)}(f_{2,2}-f_{2,3})(g_{2,2}-g_{2,3}) ] \nonumber\\
		&+ w_1 w_2 [ \sigma_{2,1}^{(2)} \sigma_{2,1}^{(1)} ((f_{2,2}-f_{2,1})-(f_{1,2}-f_{1,1}))((g_{2,2}-g_{2,1})-(g_{1,2}-g_{1,1}))\nonumber\\
		&\text{ } + \sigma_{2,3}^{(2)} \sigma_{2,1}^{(1)} ((f_{2,2}-f_{2,3})-(f_{1,2}-f_{1,3}))((g_{2,2}-g_{2,3})-(g_{1,2}-g_{1,3}))] \nonumber \\
		(f \circledcirc g)_{23} &= f_{2,3}g_{2,3}-w_1 \sigma_{2,1}^{(1)}(f_{2,3}-f_{1,3})(g_{2,3}-g_{1,3}) \nonumber\\
		&- w_2 [ \sigma_{3,1}^{(2)}(f_{2,3}-f_{2,1})(g_{2,3}-g_{2,1})+ \sigma_{3,2}^{(2)}(f_{2,3}-f_{2,2})(g_{2,3}-g_{2,2}) ] \nonumber\\
		&+ w_1 w_2 [ \sigma_{3,1}^{(2)} \sigma_{2,1}^{(1)} ((f_{2,3}-f_{2,1})-(f_{1,3}-f_{1,1}))((g_{2,3}-g_{2,1})-(g_{1,3}-g_{1,1}))\nonumber\\
		&\text{ } + \sigma_{3,2}^{(2)} \sigma_{2,1}^{(1)} ((f_{2,3}-f_{2,2})-(f_{1,3}-f_{1,2}))((g_{2,3}-g_{2,2})-(g_{1,3}-g_{1,2}))]. \nonumber 
	\end{align}
\end{ex}

\subsection{Defining the norm}

Let $M_i\subset \mathbb{C}$ be compacts for $i=1,2,$ so $M_1 \times M_2$ is a compact set. We consider continuous functions $f$ from $M_1 \times M_2\subset \mathbb{C}^{2}$ into $\mathbb{C}^{d_1\times d_2}.$ Let for short denote $X = C (M_1 \times M_2, \mathbb{C}^{d_1 \times d_2})$ and $f\in X$ having the norm
$$
|f|_\infty = \underset{j,k}{\rm max} \underset{M_1 \times M_2}{\rm max} |f_{jk}(w_1,w_2)|,
$$
for $j=1,2,\dots ,d_1,$ $k=1,2,\dots ,d_2.$
Thus $X$ is a Banach space.

\begin{prop}
	With the above setup we have
	$$ p_{j,k}= \sum_{\scriptscriptstyle \alpha_1, \alpha_2=0}^{N_1} \sum_{\scriptscriptstyle \beta_1,\beta_2=0}^{N_2} c_{\scriptscriptstyle \alpha_1 \beta_1 \alpha_2 \beta_2} w_1^{\alpha_1} \overline{w_1}^{\beta_1}  w_2^{\alpha_2} \overline{w_2}^{\beta_2} $$
	are dense in $X.$ 
\end{prop}

\begin{proof}
	Consider continuous functions on the compacts $M_1$ and $M_2$ obtained from approximating with polynomials in two variable $(w_1,\overline{w_1})$ and $(w_2,\overline{w_2}),$ i.e.
	
	\begin{align}
		p_{1,j}(w_1,\overline{w_1}) &= \sum_{\alpha_1 = 0}^{N_1} \sum_{\beta_1 =0}^{N_2} a_{\alpha_1 \beta_1} w_1^{\alpha_1} \overline{w_1}^{\beta_1} \nonumber \\
		p_{2,k}(w_2,\overline{w_2}) &= \sum_{\alpha_2 = 0}^{N_1} \sum_{\beta_2 =0}^{N_2} b_{\alpha_2 \beta_2} w_2^{\alpha_2} \overline{w_2}^{\beta_2}. \nonumber
	\end{align}
	
	Now, these polynomials stay dense in $C(M_1)$ and $C(M_2),$ respectively, due to Stone-Weierstrass theorem. 
	
	Let $f_j(w_1) \in M_1 $ and $g_k(w_2)\in M_2,$ so $f_j(w_1)g_k(w_2)\in M_1 \times M_2.$ Then Stone-Weierstrass allows to replace functions with polynomials in two variables $w_i, \overline{w_i},$ for $i=1,2.$  
	Since we have continuous functions on compact sets $M_1, M_2$ where there are finite numbers of components, and since the norm is the maximum over components, it is enough to look at continuous functions $f$ in $C(M_1 \times M_2)$ component-wise. So in $X$ an arbitrary element
	$$
	f = \sum_{j=1}^{d_1} \sum_{k=1}^{d_2} f_{j,k}(w_1,w_2) e_{1,j} \otimes e_{2,k},
	$$ 
	with $f_{j,k}$ continuous functions in $M_1 \times M_2$ can be approximated using polynomials in $(w_1, w_2),$ say $p_{j,k},$ by breaking it into finite sums of polynomials in two variables $w_1, \overline{w_1}$ and $w_2, \overline{w_2}.$ Hence the Stone-Weierstrass ensures the statement. 
\end{proof}

Now consider the multiplication in the $X$ being the one defined in the previous subsection, denoted with $\circledcirc.$ We proceed by looking at how the product $\circledcirc$ is compatible with the norm in this Banach space $X.$ Thus if we move to the operator norm we need it to be equivalent with the norm defined on $X.$

\begin{defn}\label{defnormX}
	For $f\in X$ we set
	$$
	\| f \| := \underset{|g|_\infty \leq 1}{\sup} |f \circledcirc g|_\infty .
	$$
\end{defn}

This is a norm in $X$ and it is equivalent with $|\cdot |_\infty .$

\begin{prop}
	There is a $C$, depending only on $M_1 \times M_2$ and $\Lambda_i ,$ for $i=1,2$ such that
	\begin{align}
		\|f \circledcirc g \| \leq \|f \| \|g \|, \label{eq2.17}\\
		|f|_\infty \leq \| f \| \leq C |f|_\infty .\label{eq2.18}
	\end{align}
\end{prop}

\begin{proof}
	It is clear that 
	$$
	|f|_\infty = |f \circledcirc {\bf 1} |_\infty \leq \| f\|.
	$$ 
	It is easily seen from the definition of the polyproduct as well as from \eqref{eq2.16} that there exists a constant $C$ such that 
	$$ |f \circledcirc g |_\infty \leq C |f|_\infty |g|_\infty, $$
	so we get $\|f \| \leq C |f|_\infty. $
	
	Then we have
	$$
	| f \circledcirc g \circledcirc h |_\infty \leq \| f \| |g \circledcirc h|_\infty \leq \| f \| \| g \| |h|_\infty
	$$
	which implies \eqref{eq2.17}.	
\end{proof}

\begin{defn}
	The Banach space $C(M_1\times M_2, \mathbb{C}^{d_1\times d_2})$ of continuous functions from compact $M_1\times M_2$ into $\mathbb{C}^{d_1 \times d_2}$ with the operator norm and product $\circledcirc$ becomes a Banach algebra and is denoted by $C_{\Lambda_1 \times \Lambda_2}(M_1 \times M_2 ).$
\end{defn}

A unit in the algebra is then set as follows, 
$ {\bf 1} : (w_1,w_2) \mapsto \sum_{j,k} e_{1,j} \otimes e_{2,k}.$
Here $e_{1,j}$ and $e_{2,k}$ are units in $C(M_1,\mathbb{C}^{d_1})$ and $C(M_2,\mathbb{C}^{d_2}),$ respectively.
Therefore $(X, \circledcirc , \| \cdot \| )$ is a unital Banach algebra with unit $\bf 1,$ denoted $C_{\Lambda_1 \times \Lambda_2}(M_1 \times M_2).$

Moving to algebra setup we further need to find a cross norm, say $\alpha,$ on the completed tensor algebras $ C_{\Lambda_1}(M_1) \hat{\otimes}_\alpha C_{\Lambda_2}(M_2)$ which will help in canonically identify this completed tensor algebra with the algebra $C_{\Lambda_1 \times \Lambda_2}(M_1 \times M_2).$  Hence the Gelfand theory can be applied to the multicentric representation of two commuting operators successfully. This is carefully worked out in the next section.

\section{Application of Gelfand theory}

\subsection{Identification of the Banach algebra}

Recall from \cite{ap} that if $A_1$ and $A_2$ are Banach algebras, then $A_1 \hat{\otimes}_{\varepsilon} A_2$ and $A_1 \hat{\otimes}_{\gamma} A_2$ are Banach algebras. Moreover, if $A_1$ and $A_2$ are commutative with maximal ideal spaces $\mathfrak{M}_1$ and $\mathfrak{M}_2,$ respectively, then $\mathfrak{M}_1 \times \mathfrak{M}_2$ is the maximal ideal space of both $A_1 \hat{\otimes}_{\varepsilon} A_2$ and $A_1 \hat{\otimes}_{\gamma} A_2.$

In order to apply these results to our discussion we modify the complex valued case into vector-valued case following Ryan's book \cite{sprR}.

In the complex valued case Ryan \cite{sprR} has proved that for a Banach space $Y$ and a compact set $M\subset \mathbb{C}$ the injective tensor product $C(K) \hat{\otimes}_\varepsilon Y$ can be identified with the Banach space $C(M,Y)$ and the norm on this space is given by $\|f\|_\infty =\sup \lbrace \|f(t)\|: t\in M \rbrace.$ Moreover, this identification can be applied to get a representation of a space of continuous functions of two variables as an injective tensor product of two spaces of continuous functions, that is, for $M_1, M_2$ compact spaces,
$$
C(M_1) \hat{\otimes}_\varepsilon C(M_2) = C(M_1 \times M_2),
$$
since $C(M_1 \times M_2)$ can be identified with $C(M_1,C(M_2)).$

Now let $Y$ be a Banach space and $M\subset \mathbb{C}$ a compact set. Let $\lbrace 1,2, \dots , d\rbrace  $ be an index set and let $M_d=M \times \lbrace 1,2, \dots , d \rbrace.$ In $C(M_d,\mathbb{C}),$ the continuous functions $f$ are given by $(w,j) \mapsto f_j(w)$ and the norm is the $\max$ norm. 

For simplicity we use $C(M_d)$ instead of $C(M_d,\mathbb{C}).$ We can then identify $C(M,\mathbb{C}^{d})$ with $C(M_d)$ where for $ x\in C(M, \mathbb{C}^{d})$ the norm is given by
$$
\| x \| = \max_{w\in M} \max_j |x_j(w)|
$$
and it is the same as in $C(M_d).$ 

Now recall that if $f$ is a continuous function on a closed interval, or more generally a compact set, then it is bounded and the supremum is attained by the Weierstrass extreme value theorem \cite{stew}, so we can replace the supremum by the maximum.
Hence following \cite{sprR} we have 
$$
C(M_d) \hat{\otimes}_\varepsilon Y = C(M_d,Y).
$$
Using the identification of $C(M,\mathbb{C}^{d})$ with $C(M_d,\mathbb{C})$ we see that we also have the identification of $C(M_d,Y)$ with $C(M,Y^{d}).$

Moving to two variable setup, we use the identifications of $C(M_1,\mathbb{C}^{d_1})$ with $C(M_{d_1})$ and $C(M_2,\mathbb{C}^{d_2})$ with $C(M_{d_2})$ and this gives
$$
C(M_{d_1}) \hat{\otimes}_\varepsilon C(M_{d_2}) = C(M_{d_1} \times M_{d_2})
$$
where for $i=1,2,$ $M_{d_i}= M_i \times \lbrace 1,2,\dots ,d_i \rbrace .$

Moreover, using the identifications of $C(M_{d_1},Y)$ with $C(M_1,Y^{d_1})$ and $C(M_{d_2},Y)$ with $C(M_2,Y^{d_2}),$ it is clear that we have
$$
C(M_1,Y^{d_1}) \hat{\otimes}_\varepsilon C(M_2,Y^{d_2}) = C(M_1 \times M_2, Y^{d_1 \times d_2}).
$$

Therefore we return now to our discussion and notations and we see that the following holds directly from the construction of the Banach algebra and the above discussion.

\begin{thm}
	\label{thm2.13} The Banach space $C(M_1 \times M_2, \mathbb{C}^{d_1 \times d_2})$ can be canonically identified with the injective tensor product $C(M_1, \mathbb{C}^{d_1}) \hat{\otimes}_\varepsilon C(M_2, \mathbb{C}^{d_2}),$ in short
	$$
	C(M_1 \times M_2, \mathbb{C}^{d_1 \times d_2}) = C(M_1, \mathbb{C}^{d_1}) \hat{\otimes}_\varepsilon C(M_2, \mathbb{C}^{d_2}).
	$$
\end{thm}

\subsection{Characters of the algebra }

We start by recalling a few facts from the literature \cite{tmj3}. Let $A_1$ and $A_2$ be two function algebras on compact Hausdorff spaces $M_1$ and $M_2,$ respectively, where by function algebras we mean uniformly closed subalgebras of the continuous complex-valued functions which contain the constants and separate the points.
Let $A_1 \hat{\otimes}_\varepsilon A_2$ be the completion of $A_1 \otimes_{ a} A_2$ (the algebraic tensor product of $A_1$ and $A_2$ ) under the $\varepsilon-$norm. The $\varepsilon-$norm is identical with the uniform norm on $M_1\times M_2$ and $C(M_1) \hat{\otimes}_\varepsilon C(M_2) = C(M_1\times M_2).$ Thus $A_1 \hat{\otimes}_\varepsilon A_2$ is a Banach algebra. It is easily seen that $A_1 \hat{\otimes}_\varepsilon A_2$ becomes a function algebra on $M_1 \times M_2,$ denoted by $\mathcal{U}.$

Let $\mathfrak{M}(A)$ denote the maximal ideal space of an algebra $A.$ It is known that $\mathfrak{M}(\mathcal{U})$ is homeomorphic to $\mathfrak{M}(A_1) \times \mathfrak{M}(A_2),$ see Theorem  2 in \cite{tmj}, i.e. for every $h\in \mathcal{U}$ there corresponds a unique $(\eta_1,\eta_2)\in \mathfrak{M}(A_1) \times \mathfrak{M}(A_2)$ such that $h=\eta_1 \otimes \eta_2,$ which means that if $F= \sum_{i=1}^{n} f_1 \otimes g_i \in A_1 \otimes_{ a} A_2$ then $h(F)= \sum_{i=1}^{n} \eta_1(f_i) \eta_2(g_i).$

For $F\in \mathcal{U}$ and for a fixed $y_0\in M_2$ we define $F_{y_0} $ by $F_{y_0}(x)= F(x,y_0),$ for $x\in M_1,$ and similarly is defined $F_{x_0}.$

\begin{lem}[Lemma 1 in \cite{tmj3}]
	\label{lemma 3.1} Let $F \in\mathcal{U}.$ Then every $F_{y_0}\in A_1$ and every $F_{x_0}\in A_2.$
\end{lem}

The proof, as in \cite{tmj3}, goes as follows.
\begin{proof}
	Let $\eta_{y_0}$ be the functional which associates with $g\in A_2$ the value $g(y_0).$ Then $\eta_{y_0}\in A_2^{*}$ (the dual space of $A_2$ ). We define a mapping $T_{y_0}$ of $A_1 \otimes_{ a}A_2$ into $A_1$ by $T_{y_0}(\sum_i f_i \otimes g_i) = \sum \varrho_i(y_0)f_i $ for $\sum f_i \otimes g_i \in A_1 \otimes_{ a}A_2.$ $T_{y_0}$ is continuous, so extended to $\mathcal{U}.$ Let $\lbrace \sum f_i^{(n)}\otimes g_i^{(n)} \rbrace$ be a sequence such that $ \sum f_i^{(n)}\otimes g_i^{(n)} \rightarrow F. $ Then, $T_{y_0}(F) \in A_1,$ and $(T_{y_0} (F))(x) = \lim \sum f_i^{(n)}(x) g_i^{(n)}(y_0) = F_{y_0}(x)$ for $x\in A_2,$ which completes the proof.
\end{proof}

\begin{rem}
	For a commutative Banach algebra with unit, all characters - complex homomorphisms - are automatically bounded and of norm $1.$ The focus is then on maximal ideals since these are kernels of characters.
\end{rem} 

Going back to our discussion, if we denote by $A_1$ the commutative Banach algebra with unit $C_{\Lambda_1}(M_1)$ and by $A_2$ the commutative Banach algebra with unit $C_{\Lambda_2}(M_2),$ we have that $X= A_1 \hat{\otimes}_\varepsilon A_2,$ where $X= C_{\Lambda_1 \times \Lambda_2}(M_1 \times M_2),$ is a Banach algebra with unit as defined in the previous section. Hence it is easily seen from the discussion above that we have the identification
$$
\mathfrak{M}(A_1) \times \mathfrak{M}(A_2) = \mathfrak{M}(X).
$$
In other words, one gets all the characters of the Banach algebra $X$ out of the characters of the Banach algebras $A_1$ and $A_2.$ This is just Theorem $2$ in \cite{tmj}.

\begin{thm}[Theorem $2$ in \cite{tmj}]
	Let $A_1$ and $A_2$ be commutative Banach algebras. If $A_1 \hat{\otimes}_\alpha A_2$ is a Banach algebra for the cross norm $\alpha$ not less than $\varepsilon -$norm, then $\mathfrak{M}(A_1 \hat{\otimes}_\alpha  A_2)$ is homeomorphic with the product space of $\mathfrak{M}(A_1)$ and $\mathfrak{M}(A_2).$
\end{thm}

Considering now the dual of $\mathbb{C}^{d_1 \times d_2}$ where the functionals $\eta $ are of the form
$$
\eta : a= \sum_{j=1}^{d_1} \sum_{k=1}^{d_2} \alpha_{jk} e_{1,j} \otimes e_{2,k} \mapsto \sum_j \sum_k \eta_{1,j} \eta_{2,k} \alpha_{jk}.
$$
If we require $\eta (1) = 1$ it implies that $\sum_j \sum_k \eta_{1,j} \eta_{2,k} = 1.$

Therefore, in the multicentric setup, the characters of $X$ are given by

\begin{equation}\label{eq3.1}
	\chi_{(z_1,z_2)} : f \mapsto \sum_{j=1}^{d_1} \sum_{k=1}^{d_2} \delta_{j,1}(z_1) \delta_{k,2}(z_2) f_{j,k}(p_1(z_1),p_2(z_2))
\end{equation}
and the space of all characters of $X$ will be given by the set
$$
\mathcal{X}_{X }= \lbrace \chi_{(z_1,z_2)}: (z_1,z_2)\in p_1^{-1}(M_1) \times p_2^{-1}(M_2) \rbrace.
$$

\subsection{Gelfand transform and the spectrum}

In Section $2$ we have constructed the Banach algebra $C_{\Lambda_1 \times \Lambda_2}(M_1\times M_2)$ by defining a product and an operator norm. From Proposition \ref{prop2.9} we can see that the function $\varphi$ in (\ref{eq1.2}) appears as the Gelfand transform in the algebra $C_{\Lambda_1 \times \Lambda_2 }(M_1\times M_2).$ 

When working with holomorphic function $\varphi$ is it natural to consider it as "principal" function and represented by vector function $f,$ while in a less smooth context it is better to reverse their role. Simply because $f$ can be any continuous function but the behaviour of $\varphi$ at critical points is in general complicated.

Let $\mathcal{M}(X)$ be the set of all multiplicative functionals on Banach algebra $X.$ The Gelfand transform of $f$ is defined by
$$
\Gamma : X \rightarrow \mathcal{M}(X), \quad f \mapsto \Gamma (f)
$$
where $\Gamma(f)(\eta) = \eta (f),$ for every $\eta\in\mathcal{M}.$
For $f\in X, \text{ } f$ is invertible iff $\eta (f) \neq 0$ for all $\eta \in \mathcal{M}(X) \setminus \lbrace 0 \rbrace .$ Let ${\rm Inv}(X)$ denote the set of all invertible elements in $X.$ We have
\begin{align}
	\lambda \notin \sigma (f) &\Leftrightarrow f-\lambda 1 \in {\rm Inv}(X) \nonumber\\
	&\Leftrightarrow \forall \eta \in \mathcal{M}(X)\setminus \lbrace 0 \rbrace \quad \eta(f-\lambda 1) \neq 0 \nonumber\\
	&\Leftrightarrow \forall \eta \in \mathcal{M}(X)\setminus \lbrace 0 \rbrace \quad \lambda \neq \eta(f) \nonumber\\
	&\Leftrightarrow \lambda \notin \lbrace \eta(f) : \eta \in \mathcal{M}(X)\setminus \lbrace 0 \rbrace \rbrace.
\end{align}
This means that 
$$
\sigma(f) = \lbrace \eta(f) : \eta \in \mathcal{M}(X)\setminus \lbrace 0 \rbrace \rbrace = \lbrace \Gamma(f)(\eta) : \eta \in \mathcal{M}(X)\setminus \lbrace 0 \rbrace \rbrace.
$$

Since the set of all characters of $X$ is
$
\mathcal{X}_{X}
$  defined above and 
$$
\chi_{(z_1,z_2)} (f) = \sum_{j=1}^{d_1} \sum_{k=1}^{d_2} \delta_{j,1}(z_1) \delta_{k,2}(z_2) f_{j,k}(p_1(z_1),p_2(z_2))
$$
we can identify $ \chi_{(z_1,z_2)} $ with $(z_1,z_2)$ and consequently $\mathcal{X}_{X}  $ with $ p_1^{-1}(M_1) \times p_2^{-1}(M_2). $ Hence the Gelfand transform $\hat{f}$ can be seen as a function of $(z_1,z_2)\in p_1^{-1}(M_1) \times p_2^{-1}(M_2).$
\begin{defn}
	\label{gelfandt2}
	Given $f\in C_{\Lambda_1 \times \Lambda_2} (M_1 \times M_2 )$ we set
	$$ \hat{f} : p_1^{-1}(M_1) \times p_2^{-1}(M_2) \rightarrow \mathbb{C}^{2} $$ 
	$$
	\hat{f} : (z_1,z_2) \mapsto \hat{f} ((z_1,z_2)) = \sum_{j=1}^{d_1} \sum_{k=1}^{d_2} \delta_{j,1}(z_1) \delta_{k,2}(z_2) f_{j,k}(p_1(z_1),p_2(z_2)).
	$$
\end{defn}
Hence, the multicentric representation operator $\mathcal{L}$ acts as the Gelfand transformation 
$$
\mathcal{L} : f \mapsto \hat{f}.
$$

Let $K_1\times K_2 = p_1^{-1}(M_1) \times p_2^{-1}(M_2) .$ For $f \in C_{\Lambda_1 \times \Lambda_2} (M_1 \times M_2 ) ,$ if the algebra $C_{\Lambda_1 \times \Lambda_2} (M_1 \times M_2 )$ is commutative, then the spectrum $\sigma(f)$ can be expressed in terms of the Gelfand transform
$$
\sigma(f) = \lbrace \hat{f}((z_1,z_2)): (z_1,z_2) \in K_1\times K_2 \rbrace .
$$

\subsection{Semi-simplicity property of the algebra}

Recall that an algebra $\mathcal{A}$ is {\it semi-simple} if ${\rm rad}\mathcal{A}= \lbrace 0 \rbrace .$

From \cite{tmj} we learn that a Banach space $E$ is called to satisfy \textit{the condition of approximation} if for every compact $K \subset E$ and $\varepsilon > 0 $ there exists a continuous linear mapping $u$ of finite rank on $E$ into itself such that $\|u(x)-x\|<\varepsilon$ for all $x\in K.$ The Banach algebra $C(\Omega),$ with $\Omega$ a compact Hausdorff space, satisfy this condition. Therefore we have the following result, which is Theorem $4$ in \cite{tmj}.

\begin{thm}
	\label{semisimplethm} Let $A_1$ and $A_2$ be commutative Banach algebras.
	\begin{itemize}
		\item[1.] If either $A_1$ or $A_2$ satisfy the condition of approximation then $A_1 \hat{\otimes}_\gamma A_2$ is semi-simple iff $A_1$ and $A_2$ are semi-simple.
		\item[2.] Suppose that $A_1 \hat{\otimes}_\varepsilon A_2$ becomes a Banach algebra, then $A_1 \hat{\otimes}_\varepsilon A_2$ is semi-simple if and only if $A_1$ and $A_2$ are semi-simple.
	\end{itemize}
\end{thm}
Therefore since 
$$X= C_{\Lambda_1 \times \Lambda_2 }(M_1\times M_2)$$ is a Banach algebra we get that $X$ is semi-simple iff $A_1=C_{\Lambda_1}(M_1)$ and $A_2=C_{\Lambda_2}(M_2)$ are semi-simple.

Let $M$ be compact set and $p$ a polynomial with distinct roots. The condition for $C_\Lambda(M)$ to be semi-simple is given in \cite{oam} as follows.
\begin{thm}[Theorem $2.18$ in \cite{oam}]
	$C_\Lambda(M)$ is semi-simple if and only if $M$ contains no isolated critical values of $p.$
\end{thm}

\section{Application to commuting matrices}

The multicentric calculus as in \cite{oam} provides a natural approach to deal with non-trivial Jordan blocks and one does not need differentiability at such eigenvalues.

As an application there are considered situations in which $p(A)$ is diagonalizable or similar to normal. The aim thus is to remove the Jordan blocks by moving from $A$ to $p(A).$ If $D={\rm diag}\lbrace \alpha_{j} \rbrace$ is a diagonal matrix and $\varphi$ is a continuous function, the any reasonable functional calculus satisfies $\varphi(D)={\rm diag}\lbrace \varphi(\alpha_j)\rbrace,$ while if $A$ is diagonalizable so that with a similarity $T$ one has $A=TDT^{-1},$ then one sets
\begin{equation}
	\label{eq5.1.1}
	\varphi(A)=T \varphi(D) T^{-1}.
\end{equation}

When applying the calculus in practice to $n-$tuples of commuting operators we use $n-$tuples of commuting matrices. In finding matrices which commute to a given set of $n\times n$ matrices two tools appear to be useful,  that is, a standard form for a given matrix (Jordan canonical form) and restrictions on the form of the commuting matrices. The latter is a canonical form called the H-form introduced by K.C. O'Meara and C. Vinsonhaler in \cite{laa}.

\begin{defn} A set of matrices $A_1, A_2, \dots , A_n$ is said to commute if they commute pairwise.
\end{defn}

A property of commuting matrices is that they preserve each other's eigenspaces, so they map the same invariant subspaces. If both matrices are diagonalizable, then they can be simultaneously diagonalized. Moreover, if one of the matrices has the property that its minimal polynomial coincide with its characteristic polynomial, i.e. the characteristic polynomial has only simple roots, then the other matrix can be written as a polynomial in the first matrix.

Two Hermitian matrices (or self-adjoint matrices) commute if their eigenspaces coincide. In particular, two Hermitian matrices without multiple eigenvalues commute if they share the same set of eigenvectors.

\begin{ex} \textit{(of commuting matrices).} The unit matrix commute with all matrices. Diagonal matrices commute. Jordan blocks commute with upper triangular matrices that have the same values along the diagonal and superdiagonals. If the product of two symmetric matrices is symmetric, then they must commute.
\end{ex}

Assume we are given $2$ square matrices, $A$ and $B.$ Let $A$ be diagonalizable, i.e. 
$
A= T D T^{-1}
$
and $B$ be given in the Jordan canonical form
$
B=S J S^{-1}
$
(or then assume they are given the other way around, $A$ in the Jordan canonical form and $B$ diagonalizable).

Then $A$ and $B$ commute as pointed out in the example above. One could then ask what is the minimal polynomial in each cases which makes them to be diagonalizable both. This is the case where it makes sense to consider $p_1$ and $p_2$ to be different, since they might actually not be the same.

Then those functions $\varphi$ which have $f$ in the Banach algebra structure defined in this paper, are exactly these functions which have this type of functional calculus.

\end{document}